\documentclass[12pt,reqno]{amsart}
\usepackage[a4paper,bindingoffset=0.1in,left=1.2in,right=1.3in,top=1.4in,bottom=1.5in,footskip=.25in]{geometry}
\usepackage{indentfirst,amssymb,amsfonts,amsmath,amsthm}    
\usepackage{newtxtext,newtxmath}
\usepackage{setspace}
\usepackage{times}
\usepackage[utf8]{inputenc}
\usepackage{stackengine}
\usepackage[T1]{fontenc}
\usepackage{verbatim}
\usepackage{hyperref}
\hypersetup{colorlinks=true, linkcolor=blue,citecolor=blue, urlcolor=blue}
\usepackage{mathtools}
\numberwithin{equation}{section}

\newtheorem{thm}{Theorem}[section]
\newtheorem{defin}{Definition}[section]
\newtheorem{prop}{Proposition}[section]
\newtheorem{lem}{Lemma}[section]

\newtheorem{rema}{Remark}[section]
\newtheorem*{observation}{Observation}
\newtheorem{coro}{Corollary}[section]
\newtheorem{notation}{Notation}

\DeclareMathOperator{\rk}{rk}

\DeclareMathOperator{\alt}{Alt}
\DeclareMathOperator{\tr}{Tr}
\DeclareMathOperator{\gal}{Gal}
\DeclareMathOperator{\mo}{mod}
\DeclareMathOperator{\ord}{ord}

\DeclareMathOperator{\sym}{Symm}

\makeatletter
\g@addto@macro{\endabstract}{\@setabstract}
\newcommand{\authorfootnotes}{\renewcommand\thefootnote{\@fnsymbol\c@footnote}}%
\makeatother

\begin{document}
\setcounter{page}{1} 

\baselineskip .65cm 
\pagenumbering{arabic}
\title{{Symmetric bilinear Forms and Galois Theory} }
\author [Sugata Mandal]{ Sugata Mandal}
\address{Sugata Mandal, Department of Mathematics\\
Ramakrishna Mission Vivekananda Educational and Research Institute (Belur Campus) \\
Howrah, WB 711202\\
India}
\email{gmandal1961@gmail.com\thanks{}}
\maketitle
\begin{abstract}
    Let $ K$ be a field admitting a Galois extension $L$ of degree $n$, denoting the Galois group as $G = \gal(L/K)$. Our focus lies on the space $\sym_K(L)$  of symmetric $K$-bilinear forms on $L$. We establish a decomposition of $\sym_K(L)$ into direct sum of $K$-subspaces $A^{\sigma_i}$, where $\sigma_i \in G$. Notably, these subspaces $ A^{\sigma_i}$ exhibit nice constant rank properties. The central contribution of this paper is a decomposition theorem for $\sym_K(L)$, revealing a direct sum of $\frac{(n+1)}{2}$ constant rank $n$-subspaces, each having dimension of $n$. This holds particularly when $G$ is cyclic, represented as $G = \gal(L/K) = \langle\sigma\rangle$. For cyclic extensions of even degree $n = 2m$, we present slightly less precise but analogous results. In this scenario, we enhance and enrich these constant results and show that, the component $ A^{\sigma}$ often decomposes directly into a constant rank subspaces. Remarkably, this decomposition is universally valid when $-1 \notin L^{2}$. Consequently, we derive a decomposition of $\sym_K(L)$ into subspaces of constant rank under several situations. Moreover, leveraging these decompositions, we investigate the maximum dimension of an $n$-subspace inside $M(n,K)$ and $S(n,K)$ for various field $K$ where $M(n,K)$ and $ S(n,K)$ denote the vector spaces $(n \times n)$ matrices and symmetric matrices  over $K$, respectively.

    
\noindent 
\textbf{2020 Math. Subj. Class.}: 12F05, 12F10, 15A63
\end{abstract}
\section{Introduction}
Let $K$ be a field of characteristic other than two and $\sym_K (V)$ denote the space of all symmetric bilinear forms  on a $K$-space $V$ of dimension $n$. Suppose $K$ admits a Galois extension $L$ of degree $n$. Taking the $n$-dimensional $K$-space $L$ as a model for $V$ it is shown in in this paper that ideas from Galois Theory can be fruitfully applied for studying symmetric bilinear forms on $V$. Notably, this approach sheds light on the subspaces of $\sym_K(V)$ whose nonzero skew-forms all have the same rank equal to $k$, say. Such ``$k$-subspaces" besides being interesting in their own right  play an important role in coding theory (see \cite{PD75},\cite{PD78}). An essential consideration is given to $n$-subspaces of $\sym_K(V)$-subspaces where all nonzero skew forms are non-degenerate. The Galois nature of the extension enhances the vector space structure, enabling constructions that would be unachievable without the additional field-theoretic apparatus. The  devolopment in this paper is modelled on that of \cite{GQ09} and \cite{GM2023} for skew-forms. In particular, we adapt many key results in these papers to our situation. In this paper, replacing $V$ by the $K$-space $L$, we intend to investigate symmetric bilinear forms defined on $ L \times L$ and taking values in $K$. Our particular interest lies in subspaces of such forms and their properties concerning rank. 

The Galois-theoretic trace map $\tr^L_K : L \rightarrow K$ is a central tool, defined by  \[ \tr^L_K(a) = \sum_{\sigma \in \gal(L/K)} \sigma(a), \quad \forall a \in L. \] 
 As is well known, this form is non-identically zero and exhibits the Galois invariance property $\tr(\tau(x)) = \tr(x)$ for all $x \in L$ and all $\tau \in G$. It serves as the primary tool for constructing and investigating subspaces of symmetric bilinear forms in our study.
\section{Symmetric bilinear forms and Galois extensions}
 For each $\sigma \in G : = \gal(L/K)$ and $b \in L$ we may define the symmetric bilinear form 
\begin{equation} \label{defn-sym-frm}
\phi_{b,\sigma}(x,y) = \tr^L_K(b(x\sigma(y)+\sigma(x)y)), \qquad \forall x, y \in L.  
\end{equation}
\begin{lem}
    Let $ \phi = \phi_{b,\sigma}$ be a symmetric bilinear form defined above. Thus we have \[\phi(x,y)= \tr((\sigma^{-1}(bx) + b\sigma(x)y))\] for all $x$ and $y$ in $ L$.
\end{lem}
\begin{coro}\label{radicalcondtn}
     An element $ x$ is in the radical  of $\phi_{b,\sigma}$ if and if \begin{equation}\label{radcndtn}
     \sigma^{-1}(bx) = -b \sigma(x).
     \end{equation}
\end{coro}
 Recall that if $F$ is an intermediate subfield and $a \in L$ then the $L/F$-norm  
 $N_{L/F}(a)$ of $a $ is defined as $N_{L/F}(a) = \displaystyle\prod_{\theta \in \gal(L/F)} \theta(a)$. For the sake of convenience in what follows we shall denote the fixed field of $\sigma \in \gal(L/K)$  by $L^{\langle \sigma^{i} \rangle}$.
\begin{lem}\label{degeneracy condition}
    Let $ \phi = \phi_{b,\sigma}$ be a symmetric bilinear form defined above, with $b \in L^{\times}$, then $ \phi$ is  degenerate if and only if $-\sigma(b) b^{-1}$ is expressible in the form $\sigma^{2} (c) c^{-1}$ for some  $ c \in L^{\times}$ that is,   $\phi$ is degenerate if and only if 
\begin{equation}\label{form1} 
N_{L/L^{\langle \sigma^2 \rangle}} (-\sigma(b)/b) = 1. 
\end{equation}
Moreover, in this case $\rk (\phi)= n - \frac {n}{[L:L^{\langle \sigma^2\rangle}]}$
\end{lem}
\begin{proof}
    Let $ R$ denote the radical of $\phi$ and let $x$ be an element of $R$. By Corollary \ref{radicalcondtn} and applying $\sigma$  on \eqref{radcndtn}, we obtain $$
 bx = - \sigma(b) \sigma^2(x).
 $$
 Consequently if $x\neq 0$ then taking $ x = c^{-1}$, we have  $$ -\sigma(b) b^{-1} = \sigma^{2}(c) c^{-1}.$$
 The next assertion is now clear in view of the Hilbert Theorem 90.\newline
 Now will compute the rank of $\phi$ when $\phi$ is degenerate. At first we claim that $\dim (R)= \frac {n}{[L:L^{\langle \sigma^2\rangle}]}$. Indeed let $x \in R$ then it  can be checked that $x \in c^{-1}L^{\langle \sigma^{2}\rangle}$, also conversely if $y \in c^{-1}L^{\langle \sigma^{2}\rangle} $  then $y \in R$. Thus $R = c^{-1}L^{\langle \sigma^{2}\rangle}$. Thus $\dim R = \frac {n}{[L:L^{\langle \sigma^2\rangle}]}$, that is,  $\rk(\phi)= n- \frac {n}{[L:L^{\langle \sigma^2\rangle}]}$.
\end{proof}
\begin{lem}\label{rank for odd}
    Suppose that the order of $\sigma$ is odd then for each non-zero element $ b \in L$, the symmetric bilinear form $\phi_{b,\sigma}$ is non degenerate.
\end{lem}
\begin{proof}
   If possible, let there be a $b \in L^\times$ such that $\phi_{b,\sigma}$ is degenerate. Then by Lemma \ref{degeneracy condition} we obtain a $c \in L$ satisfying $$ 
 - \sigma(b) b^{-1} = \sigma^{2} (c) c^{-1}.$$ Again by \emph{\cite[Lemma 3]{GQ09}}, $f_{b,\sigma}$ is degenerate, where $f_{b,\sigma}$ is a skew-form defined by 
 $$ f_{b,\sigma}(x,y) = \tr^L_K(b(x\sigma(y)-\sigma(x)y)), \qquad \forall x, y \in L.$$
 Consequently there exists a $d \in L$ such that $$ 
  \sigma(b) b^{-1} = \sigma^{2} (d) d^{-1},$$ by  \emph{\cite[Lemma 2]{GQ09}}. Thus $$ \sigma(cd^{-1}) = - cd^{-1},$$ which is contrary to the fact that $-1$ is an eigenvalue of $\sigma$. 
\end{proof}
\begin{observation}
In particular if $\sigma = id$ then $\phi_{b,\sigma} \neq 0 $ and $\phi_{b,\sigma} (x,y) = \tr(2bxy)$. It is non-degenerate as it is well known fact that the $\tr$ bilinear form of a  Galois extension is non-degenerate.    
\end{observation}
\begin{lem}
  Suppose that the order of $\sigma$ is even, say $2r$ then there exist $b,b'\in L^{\times}$  such that the symmetric bilinear forms $\phi_{b,\sigma}$ and $\phi_{b',\sigma}$ are degenerate and non-degenerate respectively. 
\end{lem}
\begin{proof}
 Since $\sigma$ has even 
 multiplicative order therefore $-1$ is an eigenvalue of $\sigma$, let $ b$ be a corresponding eigenvector. Consequently,  ${\sigma (b )} {b^{-1}}=-1 $. Then taking $c = 1$, we must have  $$- {\sigma (b )} {b^{-1}}= \sigma^{2}(c) c^{-1}.$$ Thus $ \phi_{b,\sigma}$ is degenerate. \newline Next we will show the existancy of a non-degenerate symmetric bilinear form. At first  we assume that $K$ is finite. Let $|L^{\langle\sigma\rangle}|= q$ then  $|L| = q^{2r}$. It is straight forward to observe that the number of elements of the form $-\sigma(b)b^{-1}$ is  $$\frac{q^{2r}-1}{q-1}$$ while the number of elements of the form $\sigma^{2}(c) c^{-1} $ is $$\frac{q^{2r} -1}{q^{2} -1}.$$ As $$ \frac{q^{2r}-1}{q^{2}-1} < \frac{q^{2r}-1}{q-1},$$  our claim follows in this case.\newline
 Now we suppose that $K$ is infinite. If possible let $\phi_{b,\sigma}$ be  degenerate $\forall b \in L$. Then by Lemma  \ref{degeneracy condition} $$N_{L/L^{\langle \sigma^2 \rangle}} (-\sigma(b)/b) = 1,$$ for all $ b \in L^{\times}$. This condition is easily seen to be equivalent to $$ (-1)^{r} b \sigma^{2}(b) \ldots \sigma^{2r-2}(b) =\sigma(b) \sigma^{3}(b)\ldots \sigma^{2r-1}(b),$$ for all $ b\in L^{\times}$. In particular, if we take any element $\alpha \in K$ and replace $ b$ above by $ \alpha - b$, we also obtain 
 $$ (-1)^{r} (\alpha -b) \sigma^{2}(\alpha -b) \ldots \sigma^{2r-2}(\alpha -b) =\sigma(\alpha-b) \sigma^{3}(\alpha -b)\ldots \sigma^{2r-1}(\alpha -b).$$ We can choose an element $b$ for which the $ 2r$ Galois conjugates  $ b, \sigma(b),\cdots,\sigma^{2r-1}(b)$ are  all different and define the polynomials $p_1$ and $p_2$ over $L$ by \begin{align*}
     p_1(x) &= (-1)^{r}(x-b)(x -\sigma^{2}(b))\ldots(x - \sigma^{2r-2}(b)),\\
      p_2(x) &= (x-\sigma(b))(x -\sigma^{2}(b))\ldots(x - \sigma^{2r-1}(b)).
 \end{align*}
 Let $p(x) = p_1(x)- p_2(x)$. Since  $p_1,p_2$ have no roots in common so $p$ is a non-zero polynomial and vanishes at every point in $K$, a contradiction. Thus there exist $b' \in L^{\times}$ such that $\phi_{b',\sigma}$ in non-degenerate.
 \end{proof}
 \begin{rema}\label{rank for even}
     If $\ord(\sigma)=2r$ then as $ b$ runs over the elements in $ L^{\times}$, the symmetric bilinear form $ \phi_{b,\sigma}$ have rank either $ n - n/r$ or $n$. \newline In particular, if $ \ord(\sigma)= 2$ then the elements $\phi_{b,\sigma}$  are either 0 or else non-degenerate. In this case   $\phi_{b,\sigma}$ is zero if and only if $b \in E$, where  $E$ be the eigenspace of $\sigma$ with respect to the eigenvalue $-1$.
 \end{rema}
 \begin{defin}
     Let $\sigma \in \gal(L/K)$. We set \[A^{\sigma}:= \{ \phi_{b,\sigma}: b \in L\}.\]
 \end{defin}
 It can be checked that $A^{\sigma}$ is a subspace of $\sym(L)$ and $ A^{\sigma}= A^{\sigma^{-1}}$. We will investigate the properties of $ A^{\sigma}$.
 \begin{thm}\label{dimension for A}
     Let $\sigma\in \gal(L/K)$. Then
     $$ 
     \dim(A^{\sigma})=
     \begin{cases}
         n, \quad  &\text{if}~\ord(\sigma) \neq 2 \\
         n/2, \quad & \text{if} ~\ord(\sigma) = 2.
     \end{cases}
     $$
     When $\sigma$ has odd order then all the nonzero elements of  $A^{\sigma}$ are non-degenerate.\newline When $ \ord(\sigma)= 2r $ then the non-zero elements of $ A^{\sigma}$ either non-degenerate or else have rank $ n-n/r$.
 \end{thm}
 \begin{proof}
     Let $\psi:L \rightarrow A^{\sigma}$ be a map define by $$ b \rightarrow \phi_{b,\sigma}.$$ Clearly $ \psi$ is $K$-linear. Let $b \in \text{ker}~ \psi$. By Corollary \ref{radicalcondtn}, $\phi_{b,\sigma}= 0$ implies $$ -bx = \sigma(b)\sigma^{2}(x),$$ for all $ x \in L$. Taking $x = 1$ we obtain $ \sigma(b) = - b$. Thus if $ b\neq 0$, we must have $$ \sigma^{2}(x)  = x$$ for all $ x \in L$. Consequently $ \ord(\sigma) = 2$ otherwise if $ \sigma =\text{id}$ then there is no non-zero $b\in L$ such that $ \sigma(b)= -b$.  Thus if $ \ord(\sigma) \neq 2$ then  $\psi$ is an isomorphism, consequently  $ \dim(A^{\sigma}) = n$. On the otherhand  if $ \ord(\sigma) = 2 $ then  $[L^{\langle \sigma \rangle} : K] = \frac{n}{2} $. Consequently the dimension of the eigenspace corresponding to the eigenvalue $-1$  is also $n/2$, thus $\dim(A^{\sigma})= \frac{n}{2}$. 
 \end{proof}
 Next we will show that when $[L:K]=n$ is odd, the space $ \sym(L)$ is the direct sum of $ (n+1)/2$, $n$-dimensional subspaces of the form $ A^{\sigma}$ where $ \sigma \in \gal(L/K)$. Suppose that $ n = 2m +1$ and $\gal(L/K)= \{1, \sigma_1,\sigma_1^{-1}, \ldots,\sigma_m,\sigma_m^{-1}\}$. \newline We shall correspondingly write $ A^{i}$ in place of $ A^{\sigma_{i}}$ and $ A^{0}:=\{\phi_{b,id}: b \in L\}$, so that $$ A ^{i} = \{ \phi_{b,\sigma_i} : b \in L\}.$$
 Let $L^{m+1}$ denote the set of ordered $(m+1)$-tuples $(x_0,x_1,\ldots,x_{m+1})$
  of elements of $L$. $ L^{m+1}$ is naturally a vector space of dimension $ n (m+1)= n (n+1)/2$ over $K$ under the usual rules  of addition of $m$-tuples and scaler multiplication.
  \begin{thm}
  \label{odd decomposition}
    Suppose that $ n = [L : K] $ is odd and the Galois group $$G = \{1,\sigma_1,\cdots,\sigma_{m}, \sigma_{1}^{-1},\cdots,\sigma_{m}^{-1}\}$$ where  $ m = (n - 1)/2 $. Then there is a direct  decomposition 
	\begin{equation}\label{odd direct decom}
	 \alt_K(L)= A^{0} \oplus A^{1}\oplus A^{2}\oplus \cdots\oplus A^{m}, 
	\end{equation}  
  Where each $A^{i}$ is an $n$-dimensional $n$-subspace of $\sym(L)$ for $ 1 \leq i \leq m $.
  \end{thm}
  \begin{proof}
     We define a $K$-linear map $$ \psi: L^{m+1} \rightarrow \sym(L)$$ by
     $$ \psi(b_0,\ldots,b_m)= \sum_{i=0}^{m} \phi_{b_i,\sigma_i}$$
  We claim that $ \psi$ is an isomorphism. Indeed let  $$ \psi(b_0,b_1,\cdots,b_m)=0,$$ then $$ \psi(b_0,b_1,\ldots,b_m)(x,y)= \sum_{i=0}^{m}\phi_{b_i,\sigma_i}(x,y)=0$$ for all $x$ and $y$. As trace form is linear it follows from Corollary \ref{radicalcondtn} that $$\sum_{i=0}^{m} \sigma_i^{-1}(b_ix) - b_i \sigma_i(x)=0$$ for all $x \in L$. Thus $$\sum_{i=0}^{m} \sigma_i^{-1}(b_i)\sigma_i^{-1} - b_i \sigma_i=0.$$  By Artin's theorem on the independence of characters we obtain $$b_0=b_1=\ldots=b_m=0.$$ Now  as $ L^{m+1}$ and $\sym(L)$ have same dimension over $ K$ thus $\psi$ is an isomorphism. Cosequently we obtain the decomposition \eqref{odd direct decom} of $\alt_K(L)$. Moreover,  it follows from the Lemma \ref{rank for odd}  that each $A^i$ is an $n$-subspace.
  \end{proof}
  \begin{thm}\label{even dcmpo}
      Suppose that $ n = [L : K] $ is even and the Galois group $$G = \{ \tau_1,\cdots,\tau_{k},1,\sigma_1,\cdots,\sigma_{m}, \sigma_{1}^{-1},\cdots,\sigma_{m}^{-1}\},$$ where $ \{\tau_{1},\tau_{2} ,\cdots,\tau_{k} \}$ are the involutions of $ G $, then there is a direct  decomposition 
	\begin{equation}\label{ev decompose}
	\alt_K(L)= B^{1}
\oplus B^{2} \oplus \cdots\oplus B^{k}\oplus A^0 \oplus A^{1} \oplus A^{2} \oplus \cdots \oplus A^{m}.    
	\end{equation}
	 where $B^{i}: =A^{\tau_i}$ is an $n$-subspace of 
  dimension $n/2$ for all $1 \le i \le k$ and   $A^{j}: = A^{\sigma_j}$ ($1 \le j \le m$) has dimension $n$. Moreover,  the rank of each element in $A^j$  is given in accordance with the results of Lemma \ref{rank for odd} and Remark \ref{rank for even}
  \end{thm}
  \begin{proof}
      We define a $K$-linear map $$ \psi: L^{k+1+m} \rightarrow \sym(L)$$ by
     $$ \psi(b_0,\ldots,b_m)= \sum_{i=1}^k \phi_{c_i,\tau_i} +\sum_{j=0}^{m} \phi_{b_i,\sigma_j}.$$ 
    By Theorem \ref{dimension for A}  the kernel of $\phi$ consists of all elements $$(b_1,\ldots,b_k,0,\ldots ,0),$$ where $  b_i \in L^{\langle\tau_i\rangle}$, $ 1 \leq i \leq k$. Consequently, $\dim\ker \psi= kn/2$.  Furthermore, since $$ \dim L^{k+m+1} - \dim \ker \psi = (k+m+1)n - \frac{kn}{2}= n(n+1)/2= \sym(L),$$ it follows that $\psi$ is surjective. 
  \end{proof}
  \section{Some Results on Cyclic Galois extension }
   In the earlier section, we have decomposed the space $ \sym_{K}(L)$ into a direct sum of constant-rank $n$-subspaces $ A^i$ when $[L:K]$ is odd. Our aim in this section is to establish the decomposition of $ \sym_K(L)$ into constant-rank subspaces, particularly when $G$ is cyclic and $n$ is even.
  \begin{notation} \label{L_i}  
\textbf{Throughout this section $L/K$ denotes a cyclic extension with Galois group $\gal(L/K) = \langle \sigma \rangle= n =2r >1$}. For the sake of convenience in what follows we shall denote the subfield $L^{\langle \sigma^{i} \rangle}$ as $L_i$.
\end{notation}
 We begin by noting the following restatement of the degeneracy criterion Lemma \ref{degeneracy condition}.
 \begin{prop}
\label{GMdegen-crit}
Let $b \in L$. Then the skew-form $\phi_{b,\sigma}$ is degenerate if and only if 
\begin{equation}\label{form1} 
(-1)^r N_{L/L_{2}} (\sigma(b)/b) = 1, 
\end{equation}
that is, $f_{b,\sigma}$ is degenerate if and only if \begin{equation} \label{form2}
N_{L/L_{2}} (b) =  b\sigma^{2}(b)\cdots \sigma^{n-2}(b) 
\end{equation}
is an  eigenvector of $\sigma$ corresponds to either the eigenvalue $-1$ or $1$ depending on whether $r$ is odd or even.
\end{prop}
\begin{proof}
    The first assertion is now clear in view of the Lemma \ref{degeneracy condition}. Moreover the condition $(-1)^r N_{L/L_{2}} (\sigma(b)/b) = 1$ is easily seen to be equivalent to the condition $$(-1)^{r}N_{L/L_2 }(b)= \sigma (N_{L/L_2 }(b)) $$ where $ N_{L/L_2 }(b)= b\sigma^{2}(b)\cdots \sigma^{n-2}(b)$. 
\end{proof}
Suppose that $\sigma^i$ is not an involution. By Lemma \ref{degeneracy condition}, the symmetric bilinear form  $\phi_{b,\sigma^i} \in A^i \subseteq \sym_K(L) $ is degenerate if and only if $-\sigma^i(b)/b= \sigma^{2i}(c)/c$ for some $ c \in L$.  As $\sigma^{2i}$ is a generator for  $\gal(L/L_{2i})$,  in view of Hilbert Theorem 90, $\phi_{b,\sigma^i}$ is degenerate if and only if  $N_{L/L_{2i} }(-\sigma^i(b)/b) = 1$. 
   A glance at Proposition  \ref{GMdegen-crit} above shows that this is precisely the condition for the symmetric bilinear form $ \displaystyle \phi^{\sim}_{b,\sigma^{i}} \in \sym_{L_i}(L)$ defined by
    \[ \phi^{\sim}_{b,\sigma^{i}} = \tr^L_{L_i} (b(x\sigma(y)+\sigma(x)y)), \qquad \forall x, y \in L. \] 
   to be degenerate (we write $\phi^{\sim}_{b,\sigma^{i}}$ instead of $\phi^{}_{b,\sigma^{i}}$ to emphasize the fact that we are now considering $L$ as $L_i$-space). 
   
   Let us write
   $\displaystyle A^{\sim 1} := \{ \displaystyle \phi^{\sim}_{b,\sigma^{i}} \mid b \in L\}$.
   In view of Theorem \ref{dimension for A} we then have a $K$-isomorphism $A^i \cong L$ via $\phi_{b,\sigma^i} \mapsto b$ and an $L_i$-isomorphism $L \cong A^{\sim 1}$ via
   $b \mapsto \phi^{\sim}_{b,\sigma^{i}}$. The composition of these maps clearly yields a $K$-isomorphism 
   $A^i \cong  A^{\sim 1}$. The following is then clear.
   \begin{rema}\label{crspn}
   With respect to the above isomorphism if an $L_i$-subspace $\mathcal W \le A^{\sim 1}$ has all its non-zero symmetric bilinear forms non-degenerate (or all its non-zero symmetric bilinear forms forms degenerate)  then the same is true for the corresponding (K)-subspace in $A^i$. 
   \end{rema}
   \begin{rema}\label{cyclic even decomposition}
    If $ L/K$ is cyclic Galois extension of degree $ n $  with $G= \gal(L/K) = \langle \sigma \rangle$ we define $A^i := A^{\sigma^{i}}$. Thus $ A^{i}= \{f_{b,\sigma^i}: b \in L\}$. If  $n$ is even  then there is a unique involution $ \tau_1 = \sigma^{n/2}$ and in this case we denote $B^1 : =A^{\tau_1}=\{f_{b,\sigma^{n/2}}: b\in L\}$. Then the decomposition \eqref{ev decompose} becomes 
\begin{equation}\label{ dcmpstn for cyclic}
\alt_K(L)=  B^{1}\oplus A^0 \oplus A^{1}\oplus A^{2}\oplus \cdots\oplus A^{m}.
\end{equation}
Note that $ B^1$ is an $n$-subspace of dimension $ n$ and by Lemma \ref{rank for odd}, $A^i$ is an $n$-subspace of dimension $n$ if $\sigma^i$ has order odd.
 \end{rema}
 
  \begin{thm} \label{odd dcmposn for any field}
   Let $K$ be a field and $n = 2k$, where  $ k\geq 1$ is odd.  Let $L$ be any cyclic  extension of $K$ of degree $n$ with Galois group $G = \langle \sigma \rangle$. 
   Then
   \begin{equation} \label{odd dcmposn for any field eq}
 A^1 = \mathcal{U}_{1} \oplus \mathcal{V}_{1},
 \end{equation}
 where    $ \mathcal{U}_{1}$ is an $n$-subspace  of dimension $k $ and $\mathcal{V}_{1}$ is an $ (n-2)$-subspace of dimension $k$.
 \end{thm}
 \begin{proof}
    Let $U := L_{k}$ and $0 \ne u \in U$. Clearly \[ \sigma^2(u), \sigma^4(u), \ldots, \sigma^{2k-2}(u) \in U. \]
It follows that \[ N_{L/L_2}(u) \in L_2 \cap U = L_2 \cap L_{k} = K. \] 
 By Proposition \ref{GMdegen-crit}   the skew-form $f_{u, \sigma}$ is non degenerate, thus $ \mathcal{U}_{1}$ is an $n$-subspace. \newline
 Note that $  N_{L/L_2}(\sigma(u)/u) = 1 $. Let  $j$ is an eigenvector  of $\sigma$ corresponding to the eigenvalue $-1$ then  $$  N_{L/L_2}(\sigma(j)/j) =- 1 . $$ By  Proposition \ref{GMdegen-crit}  $\phi_{j,\sigma}$ is degenerate. Set $ V:= jU
 $. Then for $0 \ne u \in U$
\[ N_{L/L_2}(\frac{\sigma(ju)}{ju}) = N_{L/L_2}(\frac{\sigma(j)}{j}) N_{L/L_2}(\frac{\sigma(u)}{u}) = (-1).1 =-1.\]
It thus follows by  proposition \ref{GMdegen-crit} that all the nonzero skew-forms $ \phi_{b,\sigma}$ where $b$  lies in the subspace $V = jU$ (of dimension $k$) are degenerate. Clearly $U \cap V = \{0\}$ so $L= U \oplus V$. By Theorem \ref{dimension for A} the subspace $U$ of $L $ corresponds to a subspace $\mathcal{U}_{1}$ of $\sym_K(L)$ with the same dimension defined by $\mathcal{U}_1:=\{\phi_{b,\sigma}: b \in U\}$. Similarly $V$ corresponds to $\mathcal{V}_{1} \le \sym_K(L)$ such that $\dim(V)=\dim(\mathcal{V}_1)$.
Then the decomposition \eqref{odd dcmposn for any field eq} follows.
 \end{proof}
 \begin{coro}\label{dcompsn-Ai-for-odd}
    Let $K $ be a field and $n$ be even.  Suppose $L$ is a cyclic Galois extension of a field $K$ of degree $n$ with Galois group $\gal(L/K) = \langle \sigma \rangle$. If $ \ord(\sigma^{i}) \equiv 2~(\mo 4)$  and  $ \ord(\sigma^{i}) \ne 2 $ then   \[A^i  = \mathcal{U}_i \oplus \mathcal{V}_i,\] where  $ \mathcal{U}_i$ is an $n$-subspace  of dimension $n/2 $ and $\mathcal{V}_i$ is an $ (n-2n/ \ord(\sigma^{i}))$-subspace of dimension $n/2$. 
\end{coro}

\begin{proof}
     This follows from Theorem A, noting Remark \ref{crspn} and the fact (Remark \ref{rank for even}) that a skew form in $A^i$ is either non-degenerate or has rank equal to  $n - 2n/\ord(\sigma^i)$.
\end{proof}

Consequently we obtain the following. 

\begin{coro} \label{odd dcmpsn fr any field}
     Let $K$ be a field and $n = 2k$, where  $ k\geq 1$ is odd.  Let $L$ be any cyclic Galois extension of $K$ of degree $n$ with Galois group $G = \langle \sigma \rangle$. Then
     \begin{equation}
  \alt_K(L) =  B^1 \bigoplus  A^0 \bigoplus\left(\bigoplus_{\substack{ \ord(\sigma^{i}) ~\equiv~ 0~ (\mo 2)\\\ord(\sigma^{i})\neq 2}}\left(\mathcal{U}_i \bigoplus \mathcal{V}_i \right)\right) \bigoplus  \left( \bigoplus_{\ord(\sigma^{i}) ~\equiv~ 1~ (\mo 2)} A^i \right),   
\end{equation} 
where $ B^{1}$, $A^0$, $ A^i$ (for $ \ord(\sigma^{i}) ~\equiv~ 1~ (\mo 2)$) and $\mathcal{U}_{i}$ are $n$-subspace of dimension $n/2$, $n,$ $n$ and $n/2$   respectively and $\mathcal{V}_i$ is an $ (n-2n/ \ord(\sigma^{i}))$-subspace of dimension $n/2$.
\end{coro}
\begin{proof}
In light of Remark \ref{cyclic even decomposition}, the decomposition \eqref{ dcmpstn for cyclic} becomes
 \begin{equation*}
 \alt_K(L) =  B^1 \bigoplus  A^0 \bigoplus \left(\bigoplus_{\substack{ \ord(\sigma^{i}) ~\equiv~ 0~ (\mo 2)\\\ord(\sigma^{i})\neq 2}}A^i\right) \bigoplus \left( \bigoplus_{\ord(\sigma^{i}) ~\equiv~ 1~ (\mo 2)} A^i \right).
   \end{equation*}
 
   Now it is  clear in view of Corollary \ref{dcompsn-Ai-for-odd} and Lemma  \ref{rank for odd}.
\end{proof}

In view of Theorem \ref{odd dcmposn for any field} in following theorems we focus on the case where $ n=2r$ is divisible by $ 4 $. Since in this case  $r$ is even thus by Proposition \ref{GMdegen-crit} it follows that $ \phi_{b,\sigma}$ is degenerate if and only if   \begin{equation*}
 N_{L/L_{2}} (\sigma(b)/b) = 1, 
\end{equation*}
that is,  $ 
N_{L/L_{2}} (b) $ 
is an  eigenvector of $\sigma$ corresponds to  eigenvalue $1$. This degeneracy criterion for $\phi_{b,\sigma}$ is identical to the degeneracy criterion for the skew form $f_{b,\sigma}$ (see \emph{\cite[Proposition 3.1]
{GM2023}}). 
 \begin{prop}\emph{(\cite[Lemma 3.1]{GM2023})}\label{Egnspcdecomp} 
     Let $n= 2^{\alpha}k$ where $ \alpha \geq 2$ and $k$ is odd.  Suppose that $L$ is a cyclic  extension of a field $K$ of degree $n$ with Galois group $\gal(L/K) = \langle \sigma \rangle$. Then the following hold.
     \begin{enumerate}
\item[(i)] For $1 \leq i \leq \alpha - 1 $ the subspace $ E_i:= \{b \in L : \sigma^{n/2^i}(b)= -b\} \le L$ has dimension $n/2^i$.

 \item[(ii)] Let $V_1 :=\{b \in L : \sigma^{k}(b) =  b\}$ and $V_2 :=\{b \in L : \sigma^{k}(b) =  -b\}$. Then $\dim(V_1) = \dim (V_2)= k$.
 \end{enumerate}
 \end{prop}
 Let  $\mathcal{E}_i$ be the subspace of $A^1$ corresponding to $E_i :=  \{b \in L : \sigma^{n/2^i}(b)= -b\}$ under the isomorphism of Theorem \ref{dimension for A}, that is, $\mathcal{E}_i= \{ \phi_{b,\sigma}: b \in E_i\}$ ( $1 \leq i \leq \alpha -1$ ).  Similarly, let $\mathcal V_j$ correspond to the subspace $V_j$ of $L$, that is, $ \mathcal{V}_{j}=\{ \phi_{b,\sigma}$ ($j=1,2$)\}.
\begin{thm} \emph{(\cite[Theorem B]{GM2023})} \label{ decompsn of A1 of algebraic numbere field}
  Suppose $ n = 2^{\alpha} k $ where $ \alpha \geq 2$ and $k$ is odd. Let  $ K $ be an algebraic number field such that $ -1$ is not a square in $ K$. Then there exists a cyclic  extension $L$ of $ K $ of degree $ n $ with the Galois group $G = \langle \sigma \rangle$ such that       
   \begin{equation}\label{decmpsn A1 algebraic}    
A^{1} =  \mathcal{E}_1 \oplus  \cdots  \oplus \mathcal{E}_{\alpha - 1} \oplus \mathcal{V}_1 \oplus \mathcal{V}_2 ,
  \end{equation}
   where
    \begin{itemize}
       \item[(i)] $\mathcal{E}_{i}$ is an  $n$-subspace of dimension $n/2^i$ for $ 1 \leq i \leq \alpha - 1$,
        \item[(ii)] $\mathcal{V}_j$ is an $( n-2)$-subspace of dimension $ k $ for $ 1 \leq j \leq 2$. 
    \end{itemize}
  \end{thm} 
 \begin{coro}\label{decmpsn Ai algebraic} \emph{(\cite[Corollary 4.3]{GM2023})}
In the situation of Theorem \ref{ decompsn of A1 of algebraic numbere field} if $ \ord(\sigma^{i}) \equiv 0~(\mo 4)$, say $\ord(\sigma^{i})= 2^{\beta}k' \ (\beta \ge 2)$ then   
   \begin{equation}
A^{i} = \mathcal{V}_1^i \oplus \mathcal{V}_2^i \oplus \mathcal{E}_1^i \oplus \cdots \mathcal{E}_{\beta - 1}^i,
   \end{equation}
    where
    \begin{itemize}
        \item[(i)] $\mathcal{E}_{k}^{i}$ is an  $n$-subspace of dimension $n/2^i$ for   $  1 \leq k \leq \beta - 1$,
        \item
        [(ii)]  $\mathcal{V}_j^i$ is an $(n-2)$-subspace of dimension $ k'n/{\ord(\sigma^{i})} $ for  $  1 \leq j \leq 2$.
        \end{itemize}
\end{coro}
\begin{coro}\label{algebraic dcmpsn Alt L}
                     In the situation of Theorem \ref{ decompsn of A1 of algebraic numbere field} there is direct-decomposition
                     \begin{equation}
                      \begin{aligned}[b]
            \alt_K(L) = &  B^1 \bigoplus A^0 \bigoplus \left(\bigoplus_{\substack{ \ord(\sigma^{i}) ~\equiv~ 2~ (\mo 4)\\\ord(\sigma^{i})\neq 2}} \left(\mathcal{U}_i \bigoplus \mathcal{V}_i \right)\right) \bigoplus \left( \bigoplus_{\ord(\sigma^{i}) ~\equiv~ 1~ (\mo 2)} A^i \right)  \\ &\bigoplus_{ \ord(\sigma^{i}) ~\equiv~ 0~ (\mo 4)}\left(\mathcal{V}_1^i \bigoplus \mathcal{V}_2^i \bigoplus \mathcal{E}_{\beta - 1}^i \bigoplus \cdots \bigoplus\mathcal{E}_1^i\right) 
            \end{aligned}
            \end{equation}
                 \end{coro}
                 \begin{proof}
                 
                    Using Corollaries \ref{dcompsn-Ai-for-odd}, \ref{algebraic dcmpsn Alt L} and Lemma   \ref{rank for odd}      
 as well as  the decomposition \ref{ dcmpstn for cyclic}, we can deduce the  required decomposition. 
                 \end{proof}
                 \begin{rema} \label{n-subs-exists-in}
     Let $n= 2^{\alpha}k$ where $ \alpha \geq 1$ and $k$ is odd.  Suppose that $L$ is a cyclic  extension of a field $K$ of degree $n$ with Galois group $\gal(L/K) = \langle \sigma \rangle$. If $\ord(\sigma^i)$ is even then there always exists an $n$-subspace of dimension $n/2$ inside $A^i$.
    If $ \alpha = 1$ this follows from Corollary \ref{odd dcmpsn fr any field}.
     Otherwise if $ \alpha > 1$ then
     it follows from Theorem \ref{ decompsn of A1 of algebraic numbere field} that $\mathcal{E}_1: =\{f_{b,\sigma}: b\in E_1\} $ is the desired subspace for $A^1$.  The corresponding assertion for $A^i$ now follows in the light of Corollary \ref{algebraic dcmpsn Alt L}.
 \end{rema}
                 \begin{rema}  \emph{(\cite[Remark 4.2]{GM2023})}\label{theorem B rema}
                     As its proof shows, Theorem \ref{ decompsn of A1 of algebraic numbere field} as well as its corollaries  remain valid for an arbitrary cyclic extension $L/K$ of degree $n = 2^{\alpha}k$ ($\alpha \ge 2$) such that $-1$ is not a square in $L$. Similarly, let $K$ be a field such that $f(X): = X^{4} + 1 $ is irreducible in $K[X]$ (it is not difficult to show that $K$ has this property if and only if none of $ -1 ,2$ and $-2$ is a square in $K$). Then Theorem $B$ holds true for any cyclic extension $L/K$ of degree $n = 2^\alpha k$. Indeed, if $\eta_i$ is a $2^i$-root of unity for $i \ge 1$ then the conditions $-1 \not \in K^2$ and  $\sigma(\eta_i) = -\eta_i^{-1}$ mean that $\eta_i \not \in \{-\pm 1, \pm i \}$, where $i$ denotes a primitive $4$-th root of unity in $L$. Thus $\eta$ must have order $2^s$ where $s \ge 3$. Since $\eta \in L_2$ this would mean that $L_2$ contains an element of order $8$ and thus a root of $f$ implying $f$ has a quadratic factor in $K[X]$. 
\end{rema}
                 \begin{thm} \emph{(\cite[Theorem C]{GM2023})}\label{th for finite field}
  Let $ K $  be a finite field with $ q$ elements such that $ -1$ is not a square in $ K$. Let  $ q + 1 = 2^a l$ (l odd) where $a \geq 1$ and  $ n = 2^{\alpha}k$ (k odd) where  $ \alpha \geq 2$. Suppose   $ L $ is a cyclic  extension of $ K $ of degree $ n $ with $\gal(L/K)=\langle\sigma_f\rangle$ where $\sigma_f$ is the  Frobenius map of $L$ defined by  $\sigma_f: b \rightarrow b^q$. 
         \begin{itemize}
             \item[(1)] 
 If $\alpha \leq a+1 $ then  
   \begin{equation}\label{ A1 finite field}    
A^{1} = \mathcal{V}_1 \oplus \mathcal{V}_2 \oplus \mathcal{E}_1 \oplus \cdots \oplus \mathcal{E}_{\alpha - 1},
   \end{equation}
    where
    \begin{itemize}
        \item[(i)] $\mathcal{E}_{i}$ is an  $n$-subspace of dimension $n/2^i$ for $ 1 \leq i \leq \alpha - 1$,
        \item[(ii)]  $\mathcal{V}_j$ is an $ (n-2)$-subspace of dimension $ k $ for $ 1 \leq j \leq 2$. 
    \end{itemize} 
    \item[(2)] If $\alpha >a+1 $ and $ l=1$, that is, $ q = 2^a - 1$, then 
   \begin{equation}\label{decmpsn finite field2}          
A^{1} = \mathcal{V}_1 \oplus \mathcal{V}_2 \oplus \mathcal{E}_1 \oplus \cdots \oplus \mathcal{E}_{\alpha - 1},
   \end{equation}
   where
    \begin{itemize}
        \item[(i)]   $\mathcal{E}_{i}$ is an  $n$-subspace of dimension $n/2^i$ for $ 1 \leq i \leq a$ and  an $ (n-2)$-subspace of dimension $n/2^i$ for $ a+1 \leq i \leq \alpha -1$,
        \item[(ii)]  $\mathcal{V}_j$ is an $( n-2)$-subspace of dimension $ k $ for $ 1 \leq j \leq 2$. 
    \end{itemize}
         \end{itemize}
  \end{thm}
  \begin{rema}\emph{(\cite[Remark 5.1]{GM2023})}
        In Theorem \ref{th for finite field} when $ \alpha > a+1$ and $ l > 1$ then $ \mathcal{E}_i$ is neither an $ n$-subspace nor an $( n-2)$-subspace for $a+1 \leq i \leq \alpha -1$. 
               \end{rema}
  \begin{thm}\emph{(\cite[Theorem D]{GM2023})}
   Let $p$ be a prime and $K = \mathbb{Q}_p$ be the $p$-adic  completion of $ \mathbb{Q}$ such that $ -1$ is not a square in $ K$. Let  $ p + 1 = 2^a l$ (l odd) where $a \geq 1$ and  $ n = 2^{\alpha}k$ (k odd) where  $ 2 \leq \alpha \leq a+1$. Then there exists a cyclic  extension   $ L $ of  $ K $ of degree $ n $ such that the decomposition \eqref{ A1 finite field}    holds.  
  \end{thm}
  \section{A short survey on $n$-subspace of $M_{n}(K)$ and $ S_{n}(K)$}
   For every $ n \in \mathbb{N}$ and every field $K$, let $ M(n,K)$ be the vector space of the $ (n \times n )$ matrices over $ K$, let $ S(n,K)$ be the vector space of the symmetric $ (n \times n)$ matrices over $ K$. We recall that a $K$-subpace $S$ of $M(n,K)$ or $ S(n,K)$ is an $ n $-subspace if all of its non-zero elements are invertible. Define $$ \tau_n(K) = \text{max} \{ \dim S : S ~\text{is an $n$ subspace of $M(n,K)$}\},$$   
  $$ \mu_n(K) = \text{max} \{ \dim S : S ~\text{is an $n$ subspace of $S(n,K)$}\}.$$ It is easy to observed that $ \mu_n(K) \leq \tau_n(K)$.  For an arbitrary field $ K $ exact determination of $ \tau_n(K)$ or $ \mu_{n}(K)$ is a difficult problem. However  $ \tau_{n}(K) \leq n $, as   any subspace $S\subset M_{n}(K)$ of dimension greater than $n$ must intersect nontrivially with the subspace of matrices with the first column equal to zero. 
  The invariant $\mu_{n}(K)$ is intimately relate with the invariant $ \tau_n(K)$. In general if $ n$ is even then we can easily checked that $\mu_n(K) \geq \tau_{n/2}(K)$.  Since if $U$ be an $n/2$-subspace of $M(n/2,K)$ then the subspace of all $n \times n$ symmetric matrices of the form 
  \begin{equation*}
  \begin{pmatrix}
      0 & A\\
      A^{T} &0,
  \end{pmatrix}
   \end{equation*}
   where $A$ runs over all $(n/2 \times n/2)$ matrices with entries in $U$ and $A^{T}$ denotes the transpose of $A$ is an $n$-subspace of $S(n,K)$.

  In the following, we will discuss the values of the invariants $\tau_n(K)$ and $\mu_n(K)$ across various fields.
  
      \subsection{For algebraic closed field} Suppose $K$ is an algebraic closed field. It can be easily checked that $\tau_{n}(K)=\mu_{n}(K)=1$ as for any $A$, $ B$ $ \in A_{n}(K) $, $ A - \lambda B $ is singular where $ \lambda $ is an eigen-value of $ AB^{-1} $.
      \subsection{ For real number field} Suppose $K$ is real number field. Then it follows from \emph{\cite[Theorem 1]{JPR1965}} $\tau_{n}(K)= \rho(n)$, where $ \rho(n)$ denotes the Radon-Hurtwitz number and is defined by $ \rho(n) = 2^c + 8 d$, whenever $ n = (2a+1) 2^{c +4d}$ where $ a,b,c,d$ are integers with $0\leq c \leq 3$. Now if $n$ is odd then $\rho(n)=1$, consequently $ \tau_n(K)= \mu_n(K)=1$. On the otherhand if $n$ is even then $\tau_{n/2}(K) \leq \mu_{n}(K) \leq \tau_{n}(K)$. Thus $$ \mu_{n}(K) =  8d, \quad  \text{if}~~~~~~c=0 $$
       and
      $$ 
     \mu_{n}(K)\in
     \begin{cases}
         [1+8d,2+8d] \quad  &\text{if}~~~ c=1 \\
         [2 +8d,4+8d] \quad  &\text{if}~~~ c=2 \\
         [4 + 8d,8+ 8d] \quad  &\text{if}~~~ c=3.
     \end{cases}
     $$
     \subsection{For algebraic number field}  Suppose $K$ is an algebraic number field. Then by \emph{\cite[Lemma 4]{GQ06}} for each $n$ there exists  a cyclic Galois extension $L$ of $K$. Let $ p(x)$ be a irreducible  polynomial of degree $n$ over $K$ and let $A\in M(n,K)$ be a matrix whose characteristic polynomial is $p(x)$.  Let $U\subset M(n,K)$ be the $K$-subspace spanned by the powers of $A$.  Then $U$ is an $n$-subspace of  $M(n,K)$ of dimension $n$ as $U$ is a field isomorphic to $K[x]/\bigl(p(x)\bigr)$. Thus $\tau_n(K) = n$. Although  $A$ may not necessarily belong to $S(n, K)$ but still  we can still obtain an $n$-dimensional $n$-subspace inside $ S(n,K)$. Then by Theorem 
  \ref{odd decomposition} and Remark 
 \ref{cyclic even decomposition}, $A^{0}$  is an $n$-subspace inside $ \sym_K(L)$.  Since $\sym_K(L)$ is isomorphic to $S(n,K)$, hence $A^{0}$ corresponds to an $n$-subspace of $A(n,K)$. Consequently $\mu_n(K) = n$. 
 \subsection{ For finite field} Suppose $K$ be a finite field. Then  for each $n \in \mathbb{N}$ there exists a cyclic Galois extension $ L$ of $ K$ of degree $n$. Then, employing a similar argument as in the case of an algebraic number field, we deduce $\tau_{n}(K) = \mu_{n}(K) = n $.
 \section{Conclusion}
 The eigenspaces associated with elements of the Galois group yield constant rank subspaces in $\sym_K(L)$. If $\ord(\sigma^i) \equiv 0 \pmod{2}$, then in light of the Remark \ref {n-subs-exists-in}, an $n$-subspace of dimension $n/2$ can always be found inside $A^i$. As mentioned earlier in this paper, the degeneracy criterion for the symmetric form $\phi_{b,\sigma}$ aligns with that of the skew form $f_{b,\sigma}$ (defined in \cite{GQ09}), consequently, the quest for determining the maximum dimension of an $n$-subspace inside $A^1 \subseteq \sym_K(L)$ is analogous to solving the problem of the maximum dimension of an $n$-subspace within $A^1 \subseteq \alt_K(L)$. However, spaces derived from eigenvalues may not necessarily represent the maximum possible dimension of an $n$-subspace in $A^1$ (as exemplified in \emph{\cite[Section 5]{GM2023}}), unless $n=2k$ with $k$ odd (Theorem \ref {odd dcmposn for any field}) or $K$ is finite (or more generally, $C^1$~{\emph{\cite[Lemma 3]{GQ06}}} ). Furthermore, it is noted that if $K$ admits a cyclic Galois extension of degree $n$, then $A^i$ forms an $n$-subspace of dimension $n$ (when $\sigma^i$ has odd order). The question remains open whether there exists an $n$-subspace in $\sym_K(L)$ with exactly dimension $n$ or greater than $\tau_{\frac{n}{2}}(K)$ (when $n$ is even) for a field $K$ lacking a Galois extension. Furthermore, as indicated by {\emph{\cite[Theorem 6]{GQ09}}} and {\emph{\cite[Theorem 7]{GQ09}}} for the cyclic extension $L/K$, we can infer that the rank of any non-zero symmetric form in the direct sum $A^1 \oplus A^2 \oplus \ldots \oplus A^k$ is at least $n - 2k$.
 \section*{Acknowledgements}
 The  author gratefully acknowledges support from an NBHM research award. 

\end{document}